\documentclass{amsart}

\usepackage{amsmath,amssymb, amsfonts}
\usepackage{hyperref}
\usepackage[latin1,utf8]{inputenc}
\usepackage[T1]{fontenc}

\newcommand{\tfarc}{\tfrac}
\newcommand{\ml}{\mathcal}

\newtheorem{theorem}{Theorem}
\theoremstyle{plain}
\newtheorem{corollary}{Corollary}
\newtheorem{example}{Example}

\newtheorem{proposition}{Proposition}
\newtheorem{remark}{Remark}
\numberwithin{equation}{section}
\numberwithin{example}{section}

\title{$\eta$-Normality, CR-structures, para-CR structures on almost 
contact metric, almost paracontact metric manifolds and they canonical 
connections}
\author{Piotr Dacko}
\begin{document}
\begin{abstract}
For almost contact metric or almost paracontact metric manifolds there 
is natural notion of $\eta$-normality. Manifold is called $\eta$-normal 
if is normal along kernel distribution of characteristic form. In the 
paper it is proved that $\eta$-normal manifolds are in one-one 
correspondence with Cauchy-Riemann almost contact metric manifolds or 
para Cauchy-Riemann in case of almost paracontact metric manifolds. 
There is provided characterization of $\eta$-normal manifolds in terms 
of Levi-Civita covariant derivative of structure tensor. It is 
established existence a Tanaka-like connection on $\eta$-normal 
manifold with autoparallel Reeb vector field. In particular case 
contact metric CR-manifold it is usual Tanaka connection. Similar 
results are obtained for almost paracontact metric manifolds.
For manifold with closed fundamental form we shall state uniqueness of 
this connection. In the last part is studied bi-Legendrian structure 
of almost paracontact metric manifold with contact characteristic form. 
It is established that such manifold is bi-Legendrian flat if and only 
if is normal. There are characterized semi-flat bi-Legendrian manifolds. 
\end{abstract}

\maketitle

\section{Introduction}
Almost contact metric manifolds are extensively studied in recent years.
 The study focused on contact metric manifolds however there are other 
important classes: almost cosymplectic (or almost coK\"ahler), almost 
Kenmotsu manifolds and their generalizations. For all theses classes 
were obtained similar results, for example classification of  
$(\kappa,\mu)$-spaces of different types. Besides other properties 
every such $(\kappa,\mu)$-space is Cauchy-Riemann manifold. Contact 
metric $(\kappa,\mu)$-space carries a structure of strictly 
pseudo-convex CR-manifold. While almost cosymplectic or almost Kenmotsu 
$(\kappa,\mu)$-spaces are Levi flat CR-manifolds. These results are 
cause of increasing interest in general almost contact metric 
CR-manifolds. General literature on almost contact metric manifolds are 
\cite{Blair}, \cite{CPM:DeNi:Yud}, \cite{Dileo}, \cite{GY}, \cite{O1}. 
 For almost contact metric $(\kappa,\mu)$-spaces cf. \cite{BKP}, 
 \cite{Boe}, \cite{CarMarMol}, \cite{DO2}, \cite{DO3}, \cite{Dil:Pas:2},
 \cite{PS1}. Tanaka \cite{Tanaka} introduced connection on 
non-degenerate real hypersurfce, Tanno \cite{Tan89} generalized Tanaka 
connection for arbitrary contact metric manifold. Similar problem arises
for other classes of almost contact metric manifold. To find natural 
and if possible uniquely determined connections which make almost 
contact metric structure parallel. We provide more detailed study 
of quasi-Sasakian manifolds. It is proven that each such manifold of 
non-maximal rank admits canonical foliations by cosymplectic 
leaves. There is provided family of examples of quasi-Sasakian 
structures on 2-setp nilpotent Lie groups of Heisenberg type.

In analogy to almost contact metric manifolds theory of almost 
paracontact metric manifolds was developed.  There are normal almost 
paracontact metric manifolds, although normality reads 
$[\psi,\psi]-2d\tau\otimes\zeta=0$.
Also appear almost paracontact metric $(\kappa,\mu)$-spaces. Although 
for almost paracontact metric manifolds the problem of classifying 
$(\kappa,\mu)$-spaces is far more difficult and in fact still there is 
no such classification in most interesting cases: contact para-metric, 
almost paracosymplectic or almost para-Kenmotsu.  
Early papers which treated subject in way similar to almost contact 
metric manifolds are \cite{D}, \cite{E}. In \cite{SZ} the author 
classifies almost paracontact  metric structures into classes 
determined by the decompositions of particular $\mathcal G$-module onto 
irreducible components.  For general notion of para-CR manifold, cf. 
\cite{DT}. Particular homogeneous classes of para-CR manifolds are 
studied in \cite{AMT1}, \cite{AMT2}. In framework of almost paracontact  
metric manifolds in \cite{Wel} obtained seve\-ral interesting 
conditions and characterizations for manifold to be para-CR manifold. 
Recently it was found deep relation between contact metric and 
paracontact  metric $(\kappa, \mu)$-spaces \cite{CM}, \cite{CPM:DiTer}. 
General study of paracontact  metric $(\kappa,\mu)$-spaces is provided  
in \cite{CPM:Er:Mur}. Our study of para-CR  manifolds follows the line 
of almost contact metric manifolds. We shall prove analogous 
characterization of para-CR  manifolds.  There are introduced  similar 
concepts: augmented tensor field $h$. Again we shall show existence of 
parallelizations, and uniqueness in case fundamental form is closed.

\section{Preliminaries}
All manifolds in this paper are smooth, connected, without boundary. If 
not otherwise stated we use $X$, $Y$, $Z$, ...  to denote vector fields 
on manifold. As paper treats both almost contact metric and almost 
paracontact metric structures to avoid confusion we use typographical 
convention where $\phi$, $\xi$, $\eta$ refer to almost contact 
structure while $\psi$, $\zeta$, $\tau$ - almost paracontact structure, 
resp.

\subsection{Almost contact metric manifolds}
 
Quadruple of tensor fields $(\phi,\xi,\eta, g)$, where $\phi$ is 
affinor ( $(1,1)$-tensor field), $\xi$ a vector field, $\eta$ is 
one-form, $g$ is Riemannian metric, and 
\begin{align}
& \phi^2 = -Id +\eta\otimes \xi,\quad \eta(\xi) = 1,  \\
& g(\phi X, \phi Y) = g(X,Y)-\eta(X)\eta(Y), 
\end{align}
is called almost contact metric structure. The vector field $\xi$ is 
characteristic vector field or Reeb vector field, form $\eta$ is  
characteristic form. Manifold equipped with fixed almost contact metric 
structure is called almost contact metric manifold. From definition 
tensor field $\Phi(X,Y)=g(X,\phi Y)$ is skew-symmetric -  fundamental 
form of $\mathcal M$ \footnote{In literature some authors define 
fundamental form as $g(\phi X, Y)$.}. There is $\eta \wedge \Phi^n \neq 
0$, on $\mathcal M$.

Set 
\begin{align}
& N^{(1)}(X,Y)=[\phi,\phi](X,Y)+2d\eta(X,Y)\xi,\\ 
& N^{(2)}(X,Y)=(\mathcal L_{\phi X}\eta)(Y)-(\mathcal L_{\phi Y}\eta)(X),
\end{align} 
where $\mathcal L_\xi$ denotes the Lie derivative along vector field 
$\xi$. Let $\nabla$ be covariant derivative with resp. to Levi-Civita 
connection of the metric. We have identity (cf. \cite{Blair})
\begin{align}
\label{gen:cov:div:fi}
 & 2g((\nabla_X\phi)Y,Z) = 3d\Phi(X,\phi Y, \phi Z)-3d\Phi(X,Y,Z) + 
   g(N^{(1)}(Y,Z),\phi X)+ \\
 & \qquad N^{(2)}(Y,Z)\eta(X)+2d\eta(\phi Y, X)\eta(Z)-{}
   2d\eta(\phi Z,X)\eta(Y) \nonumber .
\end{align}

Let $\mathcal D = \{\eta = 0\}$ denote the  kernel distribution of 
$\eta$. Complexification $\mathcal D^{\mathbb C}= \mathcal D'\oplus 
\mathcal D''$ splits into direct sum of complex distributions, 
$\overline{ \mathcal D'} = \mathcal D''$. If $\mathcal D'$ is formally 
involutive pair $(\mathcal M, \mathcal D')$ is called Cauchy-Riemann or 
shortly CR-manifold. Equivalently $\mathcal D'$ is formally involutive 
if and only if for vector fields $X$, $Y$, $\eta(X)=\eta(Y)=0$, there 
is vector field $Z$, $\eta(Z)=0$, such that
\begin{equation}
[X-i\phi X, Y -i\phi Y ] = Z -i\phi Z, 
\end{equation}
The Levi form $L$ of almost contact metric CR-manifold is a conformal 
equivalence class of quadratic form on $\mathcal D$, 
\begin{equation}
 -d\eta(X,\phi X), \quad \eta(X)=0.
\end{equation}
This means that in the above formula $\eta$ can be replaced by its 
multiple $f\eta$, for some smooth function $f$ on $\mathcal M$.  

It is said that almost 
contact metric manifold is $\eta$-normal if 
\begin{equation}
N^{(1)}(X,Y)=0, \quad \eta(X)=\eta(Y)=0.
\end{equation}
So $\eta$-normal manifold is manifold which is normal but only along 
kernel distribution $\{\eta =0\}$. 

\subsection{Almost paracontact  metric manifolds} 

 Almost paracontact  metric structure  is a quadruple of tensor fields 
 $(\psi, \zeta,\tau,g)$, where $\psi$ is affinor, $\zeta$ is a vector 
 field, $\tau$ is a one-form and $g$ is a pseudo-Riemannian metric. It 
 is assumed that 
\begin{align}
& \psi^2 = Id-\tau\otimes \zeta, \quad \tau(\zeta) = 1, \\
& g(\psi X,\psi Y) = -g(X,Y)+\tau(X)\tau(Y).
\end{align}
Let $\mathcal V^{\pm}$ be distributions defined by non-zero eigenvalues 
$\pm 1$ of $\psi$. From  definition dimensions of $\mathcal{V}^{\pm }$  
are equal ${\rm dim}(\mathcal V^+)={\rm dim}(\mathcal V^-)=n$, and both 
$\mathcal V^\pm$ are totally isotropic $g(\mathcal V^+,\mathcal 
V^+)=g(\mathcal V^-,\mathcal V^-)=0$.  In conclusion  pseudo-metric $g$ 
has signature $(n+1,n)$. The triple $(\psi,\zeta,\tau)$ is called 
almost paracontact  structure, tensor field $\Psi(X,Y)=g(X,\psi Y)$ is 
called  fundamental form, $\tau\wedge \Psi^n \neq 0$, everywhere on 
$\mathcal M$. Manifold equipped with fixed almost paracontact  metric 
structure is called almost paracontact  metric manifold. Almost 
paracontact  metric manifold is  para-CR manifold if eigendistributions 
$\mathcal V^\pm$, are involutive
\begin{equation}
[\mathcal V^+,\mathcal V^+] \subset \mathcal V^+, \quad 
[\mathcal V^-,\mathcal V^-] \subset \mathcal V^-.
\end{equation}

Almost paracomplex structure $J$  is $(1,1)$-tensor field satisfying 
$J^2 = Id$, and eigendistribution corresponding to eigenvalues $\pm 1$ 
are of the same dimensions. Almost paracomplex structure is said to be 
integrable if there is atlas, and in every local chart coefficients of 
$J$ are constants. It is known that sufficient and necessary condition 
for paracomplex structure $J$ to be integrable is vanishing Nijenhuis 
torsion $[J,J]=0$. In fact it is particular case of general Walker 
theorem for almost product structures.

Let $K^{(1)}= [\psi,\psi]-2d\tau\otimes \zeta$. On product $\mathcal 
M\times \mathbb S^1$ with circle there is naturally  defined almost 
paracomplex structure $J$, 
\begin{align}
\label{def:pJ}
& J(X, f\frac{d}{dt}) = (\psi X+f\zeta, \tau(X)\frac{d}{dt}),
\end{align}
 If this structure is paracomplex manifold $\mathcal M$ is said to be 
 normal. It is known that  $\mathcal M$ is normal if and only if 
 $K^{(1)} = 0$. For reader convenience we shall provide the proof of 
 this result, cf. {\bf Section \ref{sec:pcm}}.  Similarly almost 
 paracontact  metric manifold is $\tau$-normal if is normal along 
 di\-stribution $\{\tau=0\}$,
\begin{align}
K^{(1)}(X,Y)=0, \quad \tau(X)=\tau(Y)=0.
\end{align}

\section{Almost contact metric CR-manifolds and $\eta$-normal manifolds}

In this section we shall  prove  following result 

\begin{theorem}
\label{main:acm}
For almost contact metric manifold $\mathcal M$ the following 
conditions are equivalent:
\begin{enumerate}
\item[1.]
$\mathcal M$ is $\eta$-normal; 
\item[2.]
$\mathcal M$ is Cauchy-Riemann manifold; 
\item[3.]
Set $u(Y,X) = d\eta(\phi Y, X)+g(h Y,X)$. Then 
\begin{align}
\label{covd}
& g((\nabla_X\phi)Y,Z)= \tfrac{3}{2}d\Phi(X,\phi Y, \phi Z)-
\tfrac{3}{2}d\Phi(X,Y,Z) +  u(Y,X)\eta(Z) - \\ 
& \qquad u(Z,X)\eta(Y), \nonumber
\end{align}
\end{enumerate}
\end{theorem}

Here are some simple 
useful identities we use in proof of the above theorem, 
$\nabla$ denotes Levi-Civita covariant derivative operator
\begin{align}
& 2g(\nabla_X\xi, Y) = 2d\eta(X,Y)+(\mathcal L_\xi g)(X,Y), \\
& g((\nabla_X\phi)\xi,Y) = g(\nabla_X\xi,\phi Y) = d\eta(X,\phi Y)+
   \tfrac{1}{2} (\mathcal L_\xi g)(X,\phi Y),\\
& 3d\Phi(\xi,X,Y) = (\mathcal L_\xi \Phi)(X,Y), \\
& (\mathcal L_\xi \Phi)(X,Y) + (\mathcal L_\xi g)(\phi X, Y) = 
  - g ((\mathcal L_\xi\phi)X,Y) = -2g(hX,Y) \label{lxiFi}, \\
\label{n2::deta}
& N^{(2)}(X,Y) = 2 d\eta(\phi X,Y)-2d\eta(\phi Y,X), 
  \quad \eta(X)=\eta(Y)=0.
\end{align}

\begin{proposition}
Manifold is $\eta$-normal if and only is a CR-manifold. 
\end{proposition}
\begin{proof}
Let $\eta(X)=0$, $\eta(Y)=0$. For $\eta$-normal manifold
\begin{align}
\label{e-nor}
& [X,Y] - [\phi X,\phi Y] = -\phi ([\phi X,Y]+[X,\phi Y]), \\
\label{eta:fix:fiy}
&  \eta([X,Y]-[\phi X,\phi Y])=0, \\
\label{eta:x:fiy}
&  \eta([X,\phi Y]+[\phi X,Y]) =0,
\end{align}
therefore
\begin{align}
\label{fZ}
& \phi ([X,Y]-[\phi X,\phi Y]) = [\phi X, Y]+[X,\phi Y].
\end{align}
Set $Z=[X,Y]-[\phi X,\phi Y]$, $\eta(Z)=0$ by (\ref{eta:fix:fiy}), and 
the above identity implies 
\begin{align}
[X-i\phi X, Y - i \phi Y] = Z- i \phi Z,
\end{align}
so $(\mathcal M, \phi |_{\mathcal D})$ is a CR-manifold.

Conversely if manifold is CR-manifold (\ref{fZ}) is satisfied, 
consequently (\ref{e-nor}), so $\mathcal M$ is $\eta$-normal. 
\end{proof}

\begin{corollary}
For $\eta$-normal manifold 
\begin{align}
& N^{(2)}(X,Y)= 0,\quad \eta(X)= \eta(Y)=0, \\
&  d\eta(\phi X,Y)-d\eta(\phi Y,X) = 0, \quad \eta(X)=\eta(Y)=0.
\end{align}
Note these two above identities in virtue of (\ref{n2::deta}) are 
equivalent.
\end{corollary}

\begin{proposition} 
\label{covd::enor}
Almost contact metric manifold is $\eta$-normal if and only if
(\ref{covd}) holds
\end{proposition}

\begin{proof} 
Set $\bar Y = Y-\eta(Y)\xi$, $\bar Z=Z-\eta(Z)$, observe  
$\phi \bar Y = \phi Y$, $\phi \bar Z = \phi Z$, 
$\eta(\bar Y)=\eta(\bar Z)=0$. If manifold is $\eta$-normal 
$N^{(1)}(\bar Y,\bar Z)=0$,  $N^{(2)}(\bar Y,\bar Z)=0$, by 
(\ref{gen:cov:div:fi})
\begin{align}
\label{cov:div:fi:bar}
& 2g((\nabla_X\phi)\bar Y,\bar Z) = 3d\Phi(X,\phi Y, \phi Z)-
  3 d\Phi(X,\bar Y, \bar Z), \\
\label{lh}
& 2g((\nabla_X\phi)\bar Y,\bar Z) = 2g((\nabla_X\phi)Y,Z) +  
  \eta(Y)\{2d\eta(\phi Z,X)-(\mathcal L_\xi g)(\phi Z,X)\} +  \\ 
& \qquad -\eta(Z)\{2d\eta(\phi Y,X)-(\mathcal L_\xi g)(\phi Y,X)\}, 
   \nonumber \\
\label{rh}
& 3d\Phi(X,\phi Y, \phi Z)-3 d\Phi(X,\bar Y, \bar Z) = 
    3d\Phi(X,\phi Y, \phi Z) -  3d\Phi(X,Y,Z) + \\ 
& \qquad \eta(Y)(\mathcal L_\xi\Phi)(Z,X)- 
    \eta(Z)(\mathcal L_\xi\Phi)(Y,X),\nonumber
\end{align}
all together these three above identities in view of (\ref{lxiFi}) 
yield (\ref{covd}).

Conversely. Let $\eta(Y)=\eta(Z)=0$. By
\begin{align*}
& 2g(hY,Z)= g((\mathcal L_\xi\phi)Y,Z) = -(\mathcal L_\xi \Phi)(Y,Z)-
 (\mathcal L_\xi g)(\phi Y,Z) ,\\
& 2g(h Z,Y) = - 2g(h \phi Z, \phi Y) = 
  (\mathcal L_\xi \Phi)(\phi Z,\phi Y) - (\mathcal L_\xi g)( Z,\phi Y)
\end{align*}
we obtain
\begin{align}
\label{h:anti}
& g(hY,Z)-g(hZ,Y) = -\tfrac{1}{2}(\mathcal L_\xi\Phi)(Y,Z)+{}
   \tfrac{1}{2}(\mathcal L_\xi\Phi)(\phi Y,\phi Z).
\end{align}
To each term of the left hand of the identity
\begin{align}
\label{3dFi:csum}
-3d\Phi(\xi, Y,Z) = g((\nabla_\xi\phi)Y,Z)+g((\nabla_Y\phi)Z,\xi) + 
   g((\nabla_Z\phi)\xi,Y), 
\end{align}
we apply (\ref{covd}), next (\ref{h:anti}), 
in  result  equation above simplifies  to
\begin{align*}
-3d\Phi(\xi,Y,Z) = -3 d\Phi(\xi,Y,Z)+d\eta(\phi Z,Y)-d\eta(\phi Y,Z).
\end{align*}
Therefore $d\eta(\phi Y,Z)-d\eta(\phi Z,Y)=0$, equivalently  
$N^{(2)}(Y,Z)=0$. Let $X$ be arbitrary vector field. By 
(\ref{gen:cov:div:fi}), (\ref{covd}), and $N^{(2)}(Y,Z)=0$, we find 
$g(N^{(1)}(Y,Z),\phi X)=0$, in consequence $N^{(1)}(Y,Z)= 0$, and 
manifold is $\eta$-normal.
\end{proof}

\begin{corollary}
For almost contact metric manifold $\mathcal M$ the following 
conditions are equi\-valent:
\begin{enumerate}
\item
$\mathcal M$ is CR-manifold and tensor field $h$ vanishes, $h=0$;
\item
$\mathcal M$ is $\eta$-normal and tensor field $h$ vanishes, $h=0$;
\item
$\mathcal M$ is normal.
\end{enumerate}
\end{corollary}

With help of (\ref{covd}), we shall find out covariant derivatives of 
$\phi$, for some classes of almost contact metric manifolds to compare 
with already known results.

\begin{example}[Contact metric CR-manifolds] For contact metric CR-manifold 
$d\eta=\Phi$, $d\Phi =0$, tensor field $h$ is symmetric 
$g(hX,Y)=g(hY,X)$. By (\ref{covd}) 
\begin{align}
& (\nabla_X\phi)Y= g(X+hX,Y)\xi-\eta(Y)(X+hX),
\end{align}
 cf. \cite{Blair},  p. 74,  Theorem 6.6.
\end{example}

\begin{example}[Almost cosymplectic CR-manifolds] 
For almost cosymplectic mani\-fold $d\eta =0$, $d\Phi=0$, again $h$ is 
symmetric, let $AX= -\nabla_X\xi$, $A$ is symmetric, moreover  
$h=-\phi A$. By (\ref{covd}) for almost cosymplectic CR-manifold
\begin{align}
& (\nabla_X\phi)Y =-g(\phi A X,Y)\xi +\eta(Y)\phi AX.
\end{align}
Therefore almost cosymplectic manifold is CR-manifold if and only if 
manifold has K\"ahlerian leaves, cf. \cite{O3}. 
\end{example}

\begin{example}[Almost Kenmotsu CR-manifolds] 
For almost Kenmotsu manifold $d\eta=0$, $d\Phi = 2\eta\wedge \Phi$, 
tensor $h$ is symmetric. By (\ref{covd}) 
\begin{align}
& (\nabla_X\phi)Y= g(\phi X+hX,Y)\xi -\eta(Y)(\phi X+hX).
\end{align}
Therefore almost Kenmotsu manifold is CR-manifold if and only if leaves 
of the distribution $\{\eta = 0\}$  are K\"ahler, cf. \cite{Dileo}, 
\cite{Salt}. 
\end{example}

We need following identities true for every almost contact metric 
manifold 
\begin{align}
& h\phi +\phi h = \tfrac{1}{2}\mathcal L_\xi \eta \otimes \xi,  \\
\label{2etah::etanfi}
& 2\eta(hX) = \eta((\nabla_\xi\phi)X)= -g(\nabla_\xi\xi,\phi X)= 
  - (\mathcal L_\xi \eta)(\phi X). 
\end{align}
From these identities we obtain set of equivalent conditions
\begin{align}
h \phi + \phi h = 0\; \leftrightarrow \;\mathcal L_\xi\eta = 0 \;{}
  \leftrightarrow \eta\circ h =0\; \leftrightarrow 
  \nabla_\xi\xi =0.
\end{align}
Let assume $h\phi +\phi h =0$, equivalently $\nabla_\xi\xi$, or 
$\mathcal L_\xi\eta = 0$, or $\eta(hX)=0$, under this assumption 
identity  (\ref{h:anti})
\begin{align}
& g(hY,Z)-g(hZ,Y) = -\tfrac{1}{2}(\mathcal L_\xi\Phi)(Y,Z) + 
\tfrac{1}{2}(\mathcal L_\xi\Phi)(\phi Y,\phi Z),
\end{align}
is satisfied for ever vector fields $Y$,$Z$, then  by (\ref{covd}) 
\begin{align}
\label{nxfi::as:h}
g((\nabla_\xi\phi)Y,Z) = g(hY,Z)-g(hZ,Y).
\end{align}
For example manifold with closed fundamental form has tensor $h$ 
symmetric, more generally let $d\Phi = 2\alpha \wedge \Phi$, again on 
such manifold $h$ is symmetric, in particular $\nabla_\xi\phi =0$. For 
the reason of the above identity we introduce augmented tensor $\bar 
h$, by
\begin{align}
& \bar h = h-\tfrac{1}{2}\nabla_\xi\phi.
\end{align}
Even though $h$ is not symmetric in general, with augmented version 
symmetry is restored $g(\bar h Y,Z) = g(\bar h Z,Y)$.  Correspondingly 
let 
\begin{align}
\bar u(X,Y) = d\eta(\phi Y, X)+g(\bar h X,Y),
\end{align}
on $\eta$-normal manifold with $\xi$ autoparallel $\bar u$ is globally 
symmetric, $\bar u (X,Y)=\bar u(Y,X)$. Accordingly (\ref{covd}), can be 
written as
\begin{align}
& g((\nabla_X\phi)Y,Z)= \tfrac{3}{2}d\Phi(X,\phi Y, \phi Z)-
   \tfrac{3}{2}d\Phi(X,Y,Z) +  \bar u(X,Y)\eta(Z) - \\ 
& \qquad \bar u(X,Z)\eta(Y) - \tfrac{1}{2}g((\nabla_\xi\phi)X,Y)\eta(Z) +
   \tfrac{1}{2}g((\nabla_\xi\phi)X,Z)\eta(Y). \nonumber 
\end{align}
\begin{proposition}
\label{prop:acm:par}
Assume almost contact metric manifold $\mathcal M$ is a CR-manifold, 
that $h\phi +\phi h=0$, equivalently in geometric terms Reeb field is 
autoparallel with resp. to Levi-Civita connection. There exists linear 
connection $\tilde\nabla$ on $\mathcal M$, that almost contact metric 
structure is parallel with respect to $\tilde \nabla$, 
\begin{align}
& \tilde \nabla \phi = 0, \quad \tilde\nabla  \xi = 0, \quad 
\tilde\nabla \eta=0,\quad \tilde\nabla g = 0.
\end{align} 
\end{proposition}
\begin{proof}
Let define tensor fields $(1,2)$-tensor fields $F_XY$,  by 
$g(F_XY,Z)=\frac{3}{2}d\Phi(X,Y,Z)$, and $(1,1)$-tensor field $B$,  
$d\eta(X,Y)=g(X,BY)$, observe  $d\eta(\phi X, Y)=d\eta(\phi Y,X)$, 
implies $\phi B= B\phi$.  We set $ \nabla^{(1)}_XY = \nabla^{(0)}_XY - 
T^{(1)}_XY$, where $\nabla^{(0)}$ is Levi-Civita connection, and 
$T^{(1)}_XY$ is deformation tensor  
\begin{align}
& T^{(1)}_XY = -\phi F_XY -\eta(X)B Y - \eta(Y)( \phi \bar hX+B X)+  \\ 
 & \qquad \tfrac{1}{2}\eta(Y)\phi(\nabla_\xi\phi)X+ (d\eta(Y,X)-{}
   \tfrac{1}{2}(\mathcal L_\xi g)(Y,X))\xi \nonumber 
\end{align}
By (\ref{lxiFi}) 
\begin{align}
& d\eta(\phi Y,X)-\tfrac{1}{2}(\mathcal L_\xi g)(\phi Y, X)  = 
   d\eta(\phi Y,X)+g(\bar hY,X) - \\
& \qquad g(\bar hY,X)-\tfrac{1}{2}(\mathcal L_\xi g)(\phi Y,X) = 
    \bar u(Y,X) - 
   \tfrac{1}{2}(\mathcal L_\xi \Phi)(X,Y) + \nonumber \\ 
& \qquad \tfrac{1}{2}g((\nabla_\xi\phi)Y,X),\nonumber  
\end{align}
hence
\begin{align}
\label{tphi}
& T^{(1)}_X\phi Y = -\phi F_X\phi Y -\eta(X)B\phi Y +\bar u(Y,X)\xi - 
   \tfrac{1}{2}(\mathcal L_\xi\Phi)(X,Y)\xi -  \\
& \qquad \tfrac{1}{2}g((\nabla_\xi\phi)X,Y)\xi, \nonumber
\end{align}
as $\eta(F_XY)= g(\xi, F_XY) = \tfrac{3}{2}d\Phi(\xi, X,Y)= 
\tfrac{1}{2}(\mathcal L_\xi\Phi)(X,Y)$
\begin{align}
\label{phit}
& \phi T_XY = -\phi^2 F_XY -\eta(X)\phi BY-\eta(Y)(\phi^2 \bar hX + 
  \phi B  X)+ \\ 
& \qquad \tfrac{1}{2}\eta(Y)\phi^2(\nabla_\xi\phi)X =  F_XY - 
   \tfrac{1}{2}(\mathcal L_\xi\Phi)(X,Y)\xi - \eta(X)\phi B Y+ 
     \nonumber \\ 
& \qquad  \eta(Y)(\bar hX-\phi BX)-\tfrac{1}{2}\eta(Y)(\nabla_\xi\phi)X. 
 \nonumber  
\end{align}.
By (\ref{tphi}), (\ref{phit}), (\ref{2etah::etanfi}), 
$(T^{(1)}_X\phi)Y = T^{(1)}_X\phi Y - \phi T^{(1)}_XY$, 
\begin{align}
& (T^{(1)}_X\phi)Y = -\phi F_X\phi Y - F_XY + \bar u(X,Y)\xi -\\
& \qquad  \eta(Y)(\bar hX -\phi B X)- 
  \tfrac{1}{2}g((\nabla_\xi\phi)X,Y)\xi + 
  \tfrac{1}{2}\eta(Y)(\nabla_\xi\phi)X,  
\nonumber
\end{align} 
and from (\ref{covd}), (\ref{nxfi::as:h})
\begin{align}
&  (\nabla_X\phi)Y = -\phi F_X\phi Y - F_XY + \bar u(X,Y)\xi - \\ 
& \qquad  \eta(Y)( \bar h X -\phi B X ) -
      \tfrac{1}{2}g((\nabla_\xi\phi)X,Y)\xi + 
      \tfrac{1}{2}\eta(Y)(\nabla_\xi\phi)X. \nonumber
\end{align}
Similarly we verify $\nabla_X\xi  = T_X^{(1)}\xi$. Therefore
\begin{align}
\nabla^{(1)}\phi =0, \quad \nabla^{(1)}\xi= 0, \quad 
\nabla^{(1)}\eta = 0. 
\end{align}
 In general connection $\nabla^{(1)}$  
does not fulfill $\nabla^{(1)} g=0$. However 
\begin{align}
& (\nabla^{(1)}_X g)(\phi Y, \phi Z)= (\nabla^{(1)}_X g)(Y,Z), \\
& (\nabla^{(1)}_X g)(\phi Y, Z)+(\nabla^{(1)}_X g)(Y,\phi Z)=0.
\end{align}
Let define connection  
\begin{align}
\nabla^{(2)}_XY = \nabla^{(1)}_XY -T^{(2)}_XY,
\end{align}
with deformation tensor $T^{(2)}$  
\begin{align}
g(T^{(2)}_XY,Z)=-\tfrac{1}{2}(\nabla^{(1)}_Xg)(Y,Z).
\end{align}
Set $\tilde \nabla=\nabla^{(2)}$. Connection $\tilde \nabla $ has all 
required properties 
\begin{align*}
& \tilde \nabla \phi = 0, \quad \tilde\nabla  \xi = 0, \quad 
  \tilde\nabla \eta=0,\quad \tilde\nabla g = 0.
\end{align*}
\end{proof}

\begin{theorem}
Let $\mathcal M$ be $\eta$-normal manifold with closed fundamental 
form. There is unique linear connection on $\mathcal M$, that 
\begin{align}
  \tilde\nabla\phi = 0, \quad \tilde \nabla \xi =0, \quad 
  \tilde\nabla\eta = 0, \quad \tilde\nabla g = 0,
\end{align}
with torsion $S(X,Y)$, satisfying
\begin{align}
& S(\xi ,\phi Y) = -\phi S(\xi, Y), & \\
& S(X,Y) = 2d\eta(X,Y)\xi, \quad {\rm {on\; vector\; fields} }\quad 
  \eta(X)=\eta(Y)=0. &
\end{align}
\end{theorem}

We shall provide now examples where we apply above result to obtain 
connections which make parallel corresponding structures. For the sake 
to make some comparisons easier notice identity
\begin{align}
& d\eta(X,Y) -\tfrac{1}{2}(\mathcal L_\xi g)(X,Y) = -g(\nabla_Y \xi,X)= 
  -(\nabla_Y\eta)(X),
\end{align}
here $\nabla$ denotes Levi-Civita covariant derivative.

\begin{example}[Tanaka connection] 
For contact metric manifold $F_XY=0$, $\nabla_\xi \phi = 0$, $\bar h = 
h$, $B=\phi$, tensor $T_X^{(1)}Y$ reads
\begin{align}
& T_X^{(1)}Y = -\eta(X)\phi Y - \eta(Y)(\phi X+\phi h X) - 
  g(\nabla_X\xi,Y)\xi,
\end{align}
moreover on contact metric manifold $h = -\phi - \phi h$, now 
$T_X^{(1)}Y$ takes more familiar form
\begin{align}
& T_X^{(1)}Y = -\eta(X)\phi Y +\eta(Y)\nabla_X\xi -g(\nabla_X\xi,Y)\xi,
\end{align}
conection $\nabla_X^{(1)} = \nabla_X - T_X^{(1)}$ is Tanaka connection 
on contact metric CR-manifold, \cite{Tanaka}, {\rm (\cite{Blair}, 
~10.4, p.~170)}.
\end{example}

\begin{example}[Parallelization on  $\eta$-normal almost cosymplectic 
manifold] 
For $\eta$-normal almost cosymplectic manifold $F_XY=0$, 
$\nabla_\xi\phi=0$, $\bar h = h$, $B=0$, deformation tensor takes form
\begin{align}
& T_X^{(1)}Y = -\eta(Y)\phi h X-g(\nabla_X\xi,Y)\xi,
\end{align}
on almost cosymplectic manifold $\nabla\xi = -\phi h$, more symmetric 
form is
\begin{align}
& T_X^{(1)}Y = \eta(Y)\nabla_X\xi-g(\nabla_X\xi,Y)\xi.
\end{align}
We find $T^{(2)}=0$, thus connection $\tilde\nabla -T^{(1)}$, 
parallelize all structure $\tilde \nabla \phi = 0$, $\tilde \nabla \xi 
= 0$, $\tilde \nabla\eta = 0$, and $\tilde \nabla g =0$.
\end{example}

\begin{example}[Parallelization on $\eta$-normal almost Kenmotsu 
manifold] 
On almost Kenmotsu manifold 
\begin{align}
& F_XY = -\eta(X)\phi Y +\eta(Y)\phi X+ \Phi(X,Y)\xi,
\end{align}
$B=0$, $\nabla_\xi\phi  =0$, $\bar h = h$, for deformation tensor
\begin{align}
& T_X^{(1)}Y = -\eta(X)Y+\eta(Y)X -\eta(Y)\phi hX -g(\nabla_X\xi,Y)\xi,
\end{align}
observe $\nabla_X\xi = X - \phi h X - \eta(X)\xi$, in more symmetric 
form
\begin{align}
& T_X^{(1)}Y = -\eta(X)Y + \eta(Y)\nabla_X\xi - g(\nabla_X\xi,Y)\xi + 
   \eta(Y)\eta(X)\xi,
\end{align}
for tensor $T^{(2)}$ 
\begin{align}
  T_X^{(2)}Y = \eta(X)Y -\eta(X)\eta(Y)\xi,
\end{align}
by our construction connection
\begin{align}
   \tilde\nabla_XY = \nabla_XY-\eta(Y)\nabla_X\xi +g(\nabla_X\xi, Y)\xi,
\end{align}
is required connection.
\end{example}

\begin{remark}
$\eta$-Normal almost cosymplectic or almost Kennmotsu manifolds admits 
natural foliation  by K\"ahler hypersurfaces. Fixing K\"ahler leaf 
$\mathcal N_a$, the connections in above examples are extensions of 
usual Levi-Civita connection  on $\mathcal N_a$. This rises question of 
uniqueness of such extension. Another problem is following: if 
$\eta$-normal almost cosymplectic manifold or almost Kenmotsu is 
locally homogenous, in the sense of Ambrose-Singer theorem is it true 
that every K\"ahler leaf must be locally homogenous as well?  
\end{remark}

\begin{theorem} 
Let $\mathcal M$ be $\eta$-normal almost contact metric manifold with 
autoparallel Reeb field. Set of parallelizations on $\mathcal M$, ie. 
connections which make  underlying almost contact metric structure 
parallel, has natural structure of affine bundle over $\mathcal M$ 
modeled on vector bundle of $(1,2)$-tensors $T_XY$ with properties 
\begin{align}
\label{t:fi:xi:g}
T_X\phi Y = \phi T_XY, \quad T_X\xi =0, \quad g(T_XY,Z)+g(Y,T_XZ)=0.
\end{align}
\end{theorem}
\begin{proof}
First of all by {\bf Proposition \ref{prop:acm:par}}, we know set of 
parallelizations is non-empty. Let $\nabla^{(1)}$, $\nabla^{(2)}$ be 
such parallelizations, the difference $T_XY = \nabla_X^{(2)}Y - \nabla 
_X^{(1)}$, is tensor field which satisfies conditions of the theorem. 
\end{proof}

Even though just stated result is rather evident, it has been marked as 
theorem due to important consequences. It is possible now to classify 
all parallelizations on all $\eta$-normal almost contact metric 
manifolds with autoparallel Reeb field in way similar as Gray-Hervella 
or Chinea-Gonzalez classifications. Vector space of tensors 
(\ref{t:fi:xi:g}) is in natural way $\mathcal G$-module of 
corresponding Lie group $\mathcal G$, decomposition into irreducible 
$\mathcal G$-modules gives us required classification.  Below we 
provide short list of examples of $\eta$-normal almost contact metric 
manifolds, which should gives reader picture how large is class of 
$\eta$-normal almost contact metric manifolds.

\begin{example}[Products]
Product $\widetilde{  \mathcal{M} }= \mathcal M\times \mathcal H$, of 
$\eta$-normal manifold $\mathcal M$ and Hermitian manifold $\mathcal H$ 
is again $\eta$-normal manifold. 
\end{example}

\begin{example}[Hypersurfaces] 
Real hypersurface $\widetilde{\mathcal M} \subset \mathcal H$, in 
Hermitian manifold, is $\eta$-normal manifold with induced almost 
contact metric structure.
\end{example}

\subsection{Quasi-Sasakian manifolds}
Let recall definition of quasi-Sasakian manifold: it is normal  almost contact
manifold with closed fundamental form. Rank of quasi-Sasakian manifold is defined 
as $r=2p+1$, $\eta\wedge (d\eta)^p\neq 0$, $(d\eta)^{p+1}=0$, quasi-Sasakian 
manifold has maximal rank if $\eta$ is contact form, in other words rank of manifold is 
equal to its dimension, \cite{Blair67}.

\begin{theorem}
For quasi-Sasakian manifold    
\begin{align}
  & g((\nabla_X\phi)Y,Z) = d\eta(\phi Y,X)\eta(Z) - 
      d\eta(\phi Z,X)\eta(Y), \\
  & g(\nabla_X\xi,Y) = d\eta(X,Y),
\end{align}
\end{theorem}
therefore, cf. (\cite{Tanno1971}, (2.12))
\begin{align}
  & g((\nabla_X\phi)Y,Z) = -\eta(Z)(\nabla_X\eta)(\phi Y) + 
       \eta(Y)(\nabla_X\eta)(\phi Z),
\end{align}
\begin{theorem}
Let $\mathcal M$ be quasi-Sasakian manifold. Let $\mathcal C$ be 
distribution of vector fields on $\mathcal M$, that $\nabla_X\phi =0$, 
$X$ denotes section of $\mathcal C$. Then  $\mathcal C$ is invariant, 
$\phi\, \mathcal C \subset \mathcal C$, integrable $[\mathcal C, 
\mathcal C] \subset \mathcal C$, if ${dim}\, \mathcal C > 1$, every 
leaf is cosymplectic submanifold. In other words if ${dim}\, \mathcal C 
> 1$, quasi-Sasakian manifold carries natural foliation by cosymplectic 
manifolds, dimension ${ dim}\,\mathcal M - r+1$.  
\end{theorem}
\begin{proof}
Case of $\rm {dim} \,\mathcal C =1$, is trivial, $\mathcal C$ is just 
$\mathcal C = \mathbb R\xi$. Let $Y$, $Z$ be arbitrary vector fields,
 consider $\rm {dim} \,\mathcal C > 1$.
If $\nabla_X\phi=0$, then
\begin{align}
   & g((\nabla_{\phi X}\phi)Y,Z) = d\eta(\phi Y, \phi X)\eta(Z) -  
      d\eta(\phi Z,\phi X)\eta(Y) = \\
   & \qquad d\eta(\phi^2 X,Y)\eta(Z) -  d\eta(\phi^2 X, Z)\eta(Y) = 0,
     \nonumber 
\end{align}
therefore $\phi \,\mathcal C \subset \mathcal C$. Let $\nabla_W\phi=0$, 
by 
\begin{align}
  & 0 = 3(d^2\eta)(\phi Y, X,W) = \phi Yd\eta(X,W) + 
        Xd\eta(W,\phi Y) + \\
  & \qquad Wd\eta(\phi Y,X) - d\eta([\phi Y,X],W) - 
        d\eta([X,W],\phi Y) -
    \nonumber \\
  & \qquad d\eta([W,\phi Y],X]= - d\eta([X,W],\phi Y),
    \nonumber
\end{align}
we have
\begin{align}
   g((\nabla_{[X,W]}\phi)Y,Z) = d\eta(\phi Y, [X,W])\eta(Z) - 
       d\eta(\phi Z, [X,W])\eta(Y) = 0,
\end{align}
we see $\mathcal C$ is involutive $[\mathcal C,\mathcal C] \subset 
\mathcal C$, it determines foliation of $\mathcal M$, as $\xi$ is 
tangent to $\mathcal C$, leaves are invariant submanifolds. Let 
$(\tilde\phi, \tilde\xi,\tilde\eta, \tilde g)$ be an almost contact 
metric structure on a leaf $\mathcal L$, by normality of $\mathcal M$, 
Nijenhuis torsion of $\tilde\phi$ vanishes thus by Blair's theorem 
$\mathcal L$ is cosymplectic manifold. Observe that rank of 
quasi-Sasakian manifold is equal to $\rm{dim}\,C^\perp +1$.
\end{proof}

Universal cover of complete cosymplectic manifold is Riemann product of 
real line and corresponding K\"ahler manifold. Particular cases are 
tori with flat cosymplectic structure. Problem arises to study 
structure of quasi-Saskian manifold on toric bundles over K\"ahler 
manifold in the soul of Boothby-Wang theorem or Blair theorem 
concerning principal circle bundles over K\"ahler manifold in general 
case of quasi-Sasakian  manifold.

Basic invariants of quasi-Sasakian manifold is its rank and signature 
of Levi form, if signature is $(2k, 2l)$ rank of manifold is $r = 2(k + 
l) + 1$. Below we provide family of examples for every possible 
combinations of rank and signature.

\begin{example}
Let $\mathcal M = \mathbb R^{2n+1}$, $(x^i, y^i, z)$, $i=1,\ldots n$ be 
global coordinates of point, $q(v)= a_{ij}v^iv^j$, $i,j=1,\ldots n$ be 
real quadratic form on $n$-dimensional real vector space. Set 
$\eta = dz -a_{ji}y^idx^j$, $\xi =\partial/\partial_z$,
\begin{align}
  & V_k = \sqrt{2}\partial_{y^k}, \quad
    V_{n+l} = \sqrt{2}(a_{li}y^i \partial_z + \partial_{x^l}),\quad
    k,l,i=1,\ldots n,
\end{align}
define almost contact metric structure $(\phi, \xi,\eta,g)$
\begin{align}
  & \phi \xi =0,\quad \phi V_k = V_{n+k}, \quad \phi V_{n+k} = -V_k, 
    \quad k =1,\ldots n,
\end{align}
metric $g$ is defined by requirement $(\xi, V_k, V_{n+k})$ form 
orthonormal frame. Metric tensor and fundamental form in coordinates 
have simple forms 
\begin{align}
  & g = \eta\otimes\eta + 
        \frac{1}{2}\sum\limits_{i=1}^n ( dx^i\otimes dx^i + 
        dy^i\otimes dy^i), \\
  & \Phi = 2\sum\limits_{i=1}^2 dx^i\wedge dy^i,
\end{align}
for Levi form $ L = d\eta(\phi Y,X)$
\begin{align}
  & L = \frac{1}{2}\sum\limits_{i,j=1}^n a_{ij}dx^i\otimes dx^j + 
      \frac{1}{2}\sum\limits_{i,j=1}^{n}a_{ij}dy^i\otimes dy^j,
\end{align}
the only non-zero Jacobi commutators of vector fields $\xi$, $V_i$, 
$V_{n+j}$, $i,j=1,\ldots n$ are
\begin{align}
  & [V_i, V_{n+j}] =2a_{ji}\partial_z, 
\end{align}
we apply this to find $N^{(1)}=0$, structure is normal,
from coboundary formula $2d\eta(X,Y) = X\eta(Y)-Y\eta(X) - \eta([X,Y])$ 
\begin{align*}
  & d\eta(\phi V_i, V_j) = d\eta(\phi V_j, V_i) =  a_{ij}, \\
  &  d\eta(\phi V_{n+i}, V_{n+j}) = d\eta(\phi V_{n+j}, V_{n+i}) = 
     a_{ij}, \\
  & d\eta(\phi V_i, V_{n+j}) = d\eta(\phi V_{n+j}, V_i) = 0, \quad 
    i,j = 1,\ldots n,
\end{align*} Levi form $-d\eta(X,\phi Y)$ has constant signature 
$(2k,2l)$, where $(k,l)$ is signature of quadratic form $q(v)$. Vector 
fields $\xi, V_i, V_{n+i}$, $i=1,\ldots n$ form Lie algebra of 
group of Heisenberg type $\mathbb H^{2n+1}$. Therefore there exists 
left-invariant almost contact metric structure on corresponding  Lie 
group. As $\mathbb H^{2n+1}$ is nilpotent, by Malcev theorem if 
quadratic form $q(v)$ is over $\mathbb Z$, group $\mathbb H^{2n+1}$ 
admits cocompact lattice $\Gamma$. If $\Gamma \backslash \mathbb 
H^{2n+1}$ is right compact quotient left-invariant structure projects 
onto structure on compact nilmanifold $\Gamma \backslash \mathbb 
H^{2n+1}$. 
\end{example}

\begin{example}[$\alpha$-Sasakian manifolds] 
Set in above example $a_{ij} = \alpha \delta_{ij}$, $i,j=1,\ldots n$, 
$\delta_{ij}$ - Kronecker's symbol. Obtained manifold is $\alpha$-Sasakian.
\end{example}

\section{Normal and $\tau$-normal almost paracontact  metric manifolds}
\label{sec:pcm}
For the sake of completness of exposition we shall provide here some 
known concepts and constructions on almost paracontact metric 
manifolds. Here we obtain necessary and sufficient conditions for 
vanishing Nijenhuis torsion of almost paracomplex structure $J$, 
(\ref{def:pJ}) \footnote{All these results are well-known in the 
literature.}. It  suffices to compute 
$[J,J]((X,0),(Y,0))$, and $[J,J]((X,0), (0, \frac{d}{dt}))$:
\begin{align*}
& [J,J]((X,0),(Y,0))  =   ([X,Y],0) + 
   [(\psi X, \tau(X)\frac{d}{dt}), (\psi Y, \tau(Y)\frac{d}{dt})] - \\
& \qquad   J[(\psi X, \tau(X)\frac{d}{dt}), (Y,0)] - 
           J[(X,0), (\psi Y, \tau(Y)\frac{d}{dt})] = \\
& \qquad   (\psi ^2[X,Y]+ \tau([X,Y])\zeta,0) + 
     ([\psi X, \psi Y], (\psi X\tau(Y)-\psi Y\tau(X))\frac{d}{dt}) - \\
& \qquad (\psi[\psi X,Y] -Y\tau(X)\zeta, \tau([\psi X,Y])\frac{d}{dt}) - 
    (\psi [X,\psi Y] +X\tau(Y)\zeta, \tau([X,\psi Y])\frac{d}{dt}) = \\       
& \qquad ([\psi,\psi](X,Y)-2d\tau(X,Y)\zeta, 
  (\mathcal L_{\psi X}\tau(Y)-\mathcal L_{\psi Y}\tau(X))\frac{d}{dt}). 
\end{align*}
Set $K^{(1)}(X,Y) = [\psi,\psi](X,Y)-2d\tau(X,Y)\zeta$, 
$K^{(2)}(X,Y)= \mathcal L_{\psi X}\tau(Y)-\mathcal L_{\psi Y}\tau(X)$. 
Vanishing of $K^{(1)}$, $K^{(2)}$ is necessary for $J$ to be 
paracomplex. Now 
\begin{align*}
  & [J,J]((X,0),(0,\frac{d}{dt}))  =  
    [(\psi X, \tau(X)\frac{d}{dt}), (\zeta,0)]  - 
         J[(X,0), (\zeta,0)] =   \\ 
  & \qquad 
    ([\psi X, \zeta], -\zeta\tau(X)\frac{d}{dt}) -  
         (\psi [X,\zeta], \tau([X,\zeta])\frac{d}{dt}) = 
  -((\mathcal L_\zeta \psi) X, (\mathcal L_\zeta \tau)(X)\frac{d}{dt}). 
\end{align*}
Set $K^{(3)}= \mathcal L_\zeta \psi$, $K^{(4)}= \mathcal L_\zeta \tau$. 
Exactly in the same way as for almost contact  metric manifold it can 
be proven that vanishing of $K^{(1)}$ follows vanishing of $K^{(i)}$, 
$i=2,3,4$,  cf. \cite{Blair}. So we recall well-known  result

\begin{theorem}
Almost paracontact  structure is normal if and only if for Nijenhuis 
torsion $[\psi,\psi]$ we have
\begin{equation}
  [\psi,\psi]-2d\tau\otimes \zeta = 0.
\end{equation}
\end{theorem}

Let $(\mathcal M, \psi,\zeta, \tau, g)$ be an almost paracontact  
metric manifold, $\nabla$ - covariant derivative operator with resp. to  
Levi-Civita connection of $g$. 
\begin{proposition}
\label{cov:div:pfi}
  For  almost paracontact  metric  manifold 
\begin{align*}
  & 2g((\nabla_X\psi)Y, Z)  =  
      -3d\Psi(X,\psi Y, \psi Z)-3d\Psi(X,Y,Z) -  
       g(N^{(1)}(Y,Z),\psi X) + \\ 
  & \qquad 
       N^{(2)}(Y,Z)\tau(X) + 2d\tau(\psi Y,X)\tau(Z) - 
       d\tau(\psi Z,X)\tau(Y). 
\end{align*} 
\end{proposition}
\begin{proof}
Recall formula for Levi-Civita connection
\begin{align*}
  & 2g(\nabla_XY,Z) =  
     Xg(Y,Z)+Yg(X,Z)-Zg(X,Y) + g([X,Y],Z)+ \\ 
  & \qquad 
    g([Z,X],Y)+g([Z,Y],X),
\end{align*}
and coboundry formula for exterior derivative $d\Psi$
\begin{align*}
  & 3d\Psi(X,Y,Z)  =  
       X\Psi(Y,Z)+Y\Psi(Z,X)+Z\Psi(X,Y) - \Psi([X,Y],Z)- \\ 
  & \qquad 
       \Psi([Y,Z],X)-\Psi([Z,X],Y).
\end{align*}
Now 
\begin{align*}
  & 2g((\nabla_X\psi)Y,Z) =  
     2g(\nabla_X\psi Y,Z)+2g(\nabla_XY,\psi Z) =   
     Xg(\psi Y, Z)+ \\ 
  & \qquad 
     \psi Y g(X,Z)-Zg(X,\psi Y) +  g([X,\psi Y], Z) + 
     g([Z,X],\psi Y)+ \\ 
  & \qquad 
     g([Z,\psi Y], X) +  Xg(Y,\psi Z) + Yg(X,\psi Z) - 
       \psi Z g(X,Y) + \\
  & \qquad  
     g([X,Y],\psi Z) + g([\psi Z, X],Y) + g([\psi Z, Y],X) =  \\ 
  & \qquad 
     -3d\Psi (X,Y,Z) - 3d\Psi(X,\psi Y, \psi Z) - g([Y,Z],\psi X) + \\
  & \qquad 
     \psi Y (\tau(X)\tau(Z)) + \tau(Z)\tau([X,\psi Y]) + 
       g([Z,\psi Y],X) -  \psi Z (\tau(X)\tau(Y)) + \\ 
  & \qquad 
     \tau(Y)\tau([\psi Z, X]) + g([\psi Z, Y],X) - 
       \Psi([\psi Y, \psi Z],X) = - 3d\Psi(X,Y,Z) - \\ 
  & \qquad 
      3d\Psi (X,\psi Y, \psi Z) - g(\psi^2[Y,Z],\psi X) -  
        g([\psi Y, \psi Z],\psi X) + g(\psi [\psi Y,Z],\psi X) + \\ 
  & \qquad 
      g(\psi [Y,\psi Z], \psi X) - 
       \tau(X)(\tau([\psi Y, Z])  +\tau([Y,\psi Z])) + 
       \psi Y(\tau(X)\tau(Z)) - \\
  & \qquad 
      \psi Z (\tau(X)\tau(Y)) - \tau(Z)\tau([\psi Y,X])  + 
        \tau(Y)\tau([\psi Z, X]) =  \\
  & \qquad  
      -3d\Psi(X,Y,Z) - 3d\Psi(X,\psi Y,\psi Z) - 
         g(N^{(1)}(Y,Z),\psi X) + \\
  & \qquad  
       N^{(2)}(Y,Z)\tau(X) + 2d\tau(\psi Y,X)\tau(Z) - 
         2d\tau(\psi Z,X)\tau(Y).
\end{align*}
\end{proof}

\begin{example}[Para-Sasakian manifold] 
Almost paracontact  metric manifold is paracontact  if $d\tau=\Phi$. 
Normal paracontact  metric manifold is called para-Sasakian. For 
para-Sasakian manifold 
\begin{equation}
  g((\nabla_X\psi)Y,Z) =  d\tau(\psi Y,X)\tau(Z)-d\tau(\psi Z,X)\tau(Y),
\end{equation}
as $d\tau =\Psi$, above identity follows 
$(\nabla_X\psi)Y = -g(X,Y)\zeta +\tau(Y)X$. 
\end{example}

\begin{example}[Para-cosymplectic manifold] 
Almost paracontact  metric manifold  is almost para-cosymplectic, if  
$d\tau =0$, $d\Psi =0$. If additionally is   normal is said to be 
para-cosymplectic. For para-cosymplectic manifold 
$g((\nabla_X\psi)Y,Z)=0$, and 
\begin{equation}
  \nabla \psi =0.
\end{equation}
\end{example}

\begin{example}[Para-Kenmotsu manifold] 
Almost paracontact  metric manifold is called almost para-Kenmotsu if 
$d\tau= 0$, $d\Psi = 2\tau\wedge\Psi$. Normal almost para-Kenmotsu 
manifold is called para-Kenmotsu. For para-Kenmotsu manifold
\begin{align}
  & 3d\Psi(X,Y, Z)  =  2\tau(X)\Psi(Y,Z) + 2\tau(Y)\Psi(Z,X) + 
     2\tau(Z)\Psi(X,Y) , \\
  & 3d\Psi(X,\psi Y, \psi Z)  =  2\tau(X)\Psi(\psi Y, \psi Z) = 
     -2\tau(X)\Psi(Y, Z),
\end{align}
therefore
\begin{align*}
  & g((\nabla_X\psi)Y,Z) =  \Psi(Y,X)\tau(Z)-\Psi(Z,X)\tau(Y),
\end{align*}
and $(\nabla_X\psi)Y = g(\psi X, Y)\zeta -\tau(Y)\psi X$.
\end{example}

\begin{theorem}
For almost paracontact  metric manifold following statements are 
equivalent
\begin{enumerate}
\item
Manifold is $\tau$-normal ;
\item
Manifold is para-CR manifold;
\item
Set $u(X,Y)=d\tau(\psi Y,X) + g(hX,Y)$. The following identity is 
satisfied
\begin{align}
\label{cov:div:crpfi}
  & g((\nabla_X\psi)Y,Z)  = -\frac{3}{2}d\Psi(X,\psi Y,\psi Z) - 
     \frac{3}{2}d\Psi(X,Y,Z) + u(Y,X)\tau(Z) - \\
  & \qquad 
     u(Z,X)\tau(Y). 
     \nonumber  
\end{align} 
\end{enumerate}
\end{theorem}
\begin{proof}
Proof goes the same way as proof of {\bf Theorem \ref{main:acm}}. There 
is no additional difficulties here. We repeat the same steps. In the 
first part we prove that manifold is $\tau$-normal iff is para-CR 
manifold. In the second part we prove that (\ref{cov:div:crpfi}) 
characterizes $\tau$-normal manifolds, exactly in the same way as we 
have proven {\bf Proposition \ref{covd::enor}}. 
\end{proof}

\begin{example}[Paracontact para-CR manifolds]
Manifold is paracontact metric if $d\tau = \Psi$. For paracontact 
metric manifold tensor $h$ is symmetric. By (\ref{cov:div:crpfi})
\begin{align}
  & (\nabla_X\psi)Y = -g(X-hX,Y)\zeta +\tau(Y)(X-hX).
\end{align}  
In particular paracontact $(\kappa,\mu)$-space, 
is para-CR manifold, cf. \cite{CPM:Er:Mur}.
\end{example}

\begin{example}[Paracosymplectic para-CR manifolds] 
For paracosymplectic mani\-fold $d\tau=0$, $d\Psi=0$, tensor $h$ is 
symmetric, set  $X \mapsto AX = -\nabla_X\zeta$, $A$ is $(1,1)$-tensor 
field and $h=A\psi = -\psi A$. For paracosymplectic CR-manifold  
\begin{align}
  & (\nabla_X\psi)Y = g(hX,Y)\tau(Z) - \tau(Y)hX = g(A\psi X,Y)\zeta - 
      \tau(Y)A\psi X.
\end{align}
We can state that paracosymplectic manifold is CR-manifold if and only 
if it has para-K\"ahlerian leaves, cf. \cite{D}. 
\end{example}

\begin{example}[Almost para-Kenmotsu para-CR manifolds] 
For almost para-Kenmotsu manifold $d\tau=0$, $d\Psi = 2\tau\wedge\Psi$, 
tensor field $h$ is symmetric. By (\ref{cov:div:crpfi})
\begin{align}
  & g((\nabla_X\psi)Y,Z) = g(\psi X+hX,Y)\zeta -\tau(Y)(\psi X+hX). 
\end{align}
Manifolds is para-CR if and only if it has para-K\"ahlerian leaves.
\end{example}

We can provide similar discussion as in case of almost contact metric 
manifolds. Thus manifold has  autoparallel Reeb field $\zeta$, 
equivalently 
\begin{align}
  h \psi + \psi h = 0 \leftrightarrow \mathcal L_\zeta\tau = 0 
    \leftrightarrow \tau\circ h =0,
\end{align}
and
\begin{align}
\label{nzpsi::as:h}
  & g((\nabla_\zeta\psi)Y,Z) = g(hY,Z)-g(hZ,Y), \\
  & g(hY,Z) - g(hZ,Y) = -\tfrac{1}{2}(\mathcal L_\zeta\Psi)(Y,Z) -
      \tfrac{1}{2}(\mathcal L_\zeta\Psi)(\psi Y,\psi Z).
\end{align}
Again  we introduce augmented tensor $\bar h$, resp. we set tensor 
$\bar u$
\begin{align}
  & \bar h = h - \tfrac{1}{2}\nabla_\zeta\psi, \\
  & \bar u(Y,X) = d\tau(\psi Y, X) + g(\bar hY,X),
\end{align}
tensors $\bar h$, $\bar u$, are symmetric $g(\bar h Y,Z) = g(\bar h 
Z,Y)$, $\bar u(X,Y)=u(Y,X)$ on $\tau$-normal manifold with 
$\nabla_\zeta\zeta =0$, and (\ref{cov:div:crpfi}) 
\begin{align}
\label{covx:psi:y:z}
  &  g((\nabla_X\psi)Y,Z) = -\tfrac{3}{2}d\Psi(X,\psi Y,\psi Z) - 
      \tfarc{3}{2}d\Psi(X,Y,Z) + \bar u(Y,X)\tau(Z) - \\
  & \qquad 
     \bar u(Z,X)\tau(Y) + \tfarc{1}{2}g((\nabla_\zeta\psi)Y,X)\tau(Z) - 
       \tfarc{1}{2}g((\nabla_\zeta\psi)Z,X)\tau(Y), 
    \nonumber
\end{align}

We can state similar proposition to {\bf Proposition 
\ref{prop:acm:par}.} 
\begin{proposition}
Assume almost paracontact metric manifold $\mathcal M$ is a para-CR 
mani\-fold (eq. $\tau$-normal) with autoparallel Reeb vector field. 
There exists linear connection $\tilde\nabla$ on $\mathcal M$, that 
almost paracontact metric structure is parallel with respect to $\tilde 
\nabla$, 
\begin{align}
  & \tilde\nabla\psi = 0, \quad \tilde\nabla\zeta = 0, \quad 
    \tilde\nabla\tau=0,\quad \tilde\nabla g = 0.
\end{align} 
\end{proposition}
\begin{proof}
We repeat the same steps as in the proof of {\bf Proposition 
\ref{prop:acm:par}.} Thus we shall not provide all details, just the 
first deformation tensor field $T^{(1)}$,   
\begin{align}
  & T^{(1)}_XY = \psi P_XY + \tau(X)BY - \tau(Y)(\psi \bar hX-BX) -  \\ 
  & \qquad 
     \tfrac{1}{2}\tau(Y)\psi(\nabla_\zeta\psi)X + 
     (d\tau(Y,X) - \tfrac{1}{2}(\mathcal L_\zeta g)(Y,X))\zeta, 
    \nonumber 
\end{align}
here again $P_XY$, $B$ are defined  as in {\bf Proposition 
\ref{prop:acm:par}}, $g(P_XY, Z) = \tfarc{3}{2}d\Psi(X,Y,Z)$,  $g(X,BY) 
= d\tau(X,Y)$. We find by definition of $T$ and by (\ref{covx:psi:y:z})
\begin{align}
  &   (T_X^{(1)}\psi)Y =  \psi P_X\psi Y - P_XY + \bar u(Y,X)\zeta + 
      \tau(Y)(\bar h -\psi BX) - \\
  & \qquad 
      \tfarc{1}{2}((\nabla_\zeta\psi)X,Y)\zeta + 
        \tfrac{1}{2}\tau(Y)(\nabla_\zeta\psi)X, \nonumber \\
  & (\nabla_X\psi)Y =  \psi P_X\psi Y - P_XY + \bar u(Y,X)\zeta + 
       \tau(Y)(\bar h -\psi BX) - \\
  & \qquad \tfarc{1}{2}((\nabla_\zeta\psi)X,Y)\zeta + 
       \tfrac{1}{2}\tau(Y)(\nabla_\zeta\psi)X, 
     \nonumber
\end{align}
for $\nabla^{(1)}=\nabla-T^{(1)}$,  we have $\nabla^{(1)}\psi=0$,   
$\nabla^{(1)}\zeta =0$, $\nabla^{(1)}\tau = 0$. If $\nabla^{(1)}g \neq 
0$ we rectify this by modifying $\nabla^{(1)}$ with tensor $T^{(2)}$,  
$g(T_X^{(2)}Y,Z) = -\tfrac{1}{2}(\nabla^{(1)}_Xg)(Y,Z)$, $\tilde \nabla 
= \nabla^{(1)}- T^{(2)}$.
\end{proof}

Below we provide list of examples of $\tau$-normal almost paracontact 
metric manifolds.
\begin{example}
	Product of $\tau$-normal manifold and para-Hermitian manifold is 
	again $\tau$-normal almost paracontact metric manifold.
\end{example} 

\begin{example}
	Non-degenerate hypersurface of para-Hermitian manifold is 
	$\tau$-normal almost paracontact metric manifold.
\end{example}

Many results concerning almost contact metric manifolds can be 
restated for almost paracontact metric manifolds. In particular 
set of parallelizations on $\tau$-normal almost paracontact 
metric manifold with autoparallel Reeb field has a structure of 
affine bundle. 

\subsection{Bi-Legendrian manifolds} 
Let $\ml M$ be $\tau$-normal almost paracontact metric  manifold with 
contact characteristic form. In general Reeb vector field $R$ of 
contact form $\tau$ is different from characteristic vector field 
$\zeta$, $R \neq \zeta$. Recall that $R$ is defined by conditions 
$\tau(R) = 1$, $\iota_Rd\tau =0$. For the purpose of this section we 
make assumption that always $R =\zeta$, which greatly simplifies some 
considerations. We have list of equivalent conditions
\begin{itemize}
\item[a)] 
   $R=\zeta$ ; 
\item[b)] 
  $\ml L_\zeta\tau =0$ ;
\item[c)] 
  vector field $\zeta$ is autoparallel $\nabla_\zeta\zeta=0$; 
\item[d)] tensors $h$, $\psi$ anti-commute $h\psi + \psi h =0$.
\end{itemize}
$\tau$-Normal almost paracontact metric manifold $\ml M$ carries 
natural structure of bi-Legendrian manifold:  pair of transversal 
Legendrian foliations $(\ml F_+$, $\ml F_{-})$, is determined by 
eigendistributions of structure tensor $\psi$. In what will follow we 
adopt convention to denote $X^+$ if vector field $X$ satisfies $\psi X 
=X$, similarly $X^-$, if $\psi X^- = -X^-$. With foliations $\ml 
F_{\pm}$, we can associate they Pang invariants \cite{CM}, \cite{Pang}
\begin{align}
  & \Pi_+(X^+,Y^+) = (\ml L_{X^+}\ml L_{Y^+}\tau)(\zeta) = 
     -\tau([[\zeta,X^+],Y^+]), \\ 
  &  \Pi_{-}(X^-,Y^-) = (\ml L_{X^-}\ml L_{Y^-}\tau)(\zeta) = 
    -\tau([[\zeta,X^-],Y^-]),  
\end{align} 
Note as $[[\zeta,X^+],Y] = [(\ml L_\zeta\psi)X^+,Y] + 
  [\psi [\zeta,X^+],Y]$, there is 
\begin{align}
\label{pi2plus}
  & \Pi_+(X^+,Y^+) = -\tau([[\zeta,X^+],Y^+]) = 4d\tau(hX^+,Y^+) + \\
  &    \qquad   2d\tau(\psi[\zeta,X^+],Y^+),
       \nonumber
\end{align}
moreover for para-CR manifold
\footnote{Remind  $R=\zeta$} 
\begin{align}
\label{dtpsi::-pi}
  & 2d\tau(\psi [\zeta, X^+], Y^+) = -2d\tau([\zeta,X^+],Y^+) = 
      -\Pi_{+}(X^+,Y^+),
\end{align}
combining (\ref{pi2plus}), (\ref{dtpsi::-pi}) we obtain
\begin{align}
  & \Pi_+(X^+,Y^+) = 2d\tau(hX^+,Y^+),
\end{align}
providing similar computations for $X^-$, $Y^-$ 
\begin{align}
  & \Pi_{-}(X^-,Y^-) = -2 d\tau(hX^-,Y^-).
\end{align}
Set $\chi(X,Y)=  d\tau(h X,\psi Y)$. Note useful 
identity $d\tau(h X,\psi Y) = -d\tau(\psi h X,Y) = d\tau(h\psi X,Y)$. 
By definition of $\chi$  
\begin{align}
\label{pg:inv}
  & \Pi_{+} = 2\chi|_{\ml V^{+}}, \quad \Pi_{-} = 2\chi |_{\ml V^{-}}.
\end{align}

\begin{theorem} 
Let $\ml M$ be  $\tau$-normal almost paracontact metric para-CR 
manifold, with contact characteristic form and Reeb vector field equal 
to characteristic vector field.  $\ml M$ is flat as 
bi-Legendrian manifold, if and only if is normal.
\end{theorem}
\begin{proof}
The proof ends if we will show  that $\chi= 0$, everywhere. 
Assume $\ml M$ is flat.
By (\ref{pg:inv}), we have $\chi |_{\ml V^\pm} =0$. Note 
$ \chi(X^+, Y^-) = -d\tau(hX^+,Y^-) $, from $h\psi + \psi h = 0$, 
tensor $h$ as linear map flips eigendistributions 
\begin{align}
h : \ml V^+ \rightarrow \ml V^-, \quad h : \ml V^- \rightarrow \ml V^+,
\end{align}
hence both $hX^+$, $Y^-$ are sections of Legendrian distribution $\ml 
V^-$, therefore
\begin{align}
  & \chi(X^+, Y^-) =- d\tau(hX^+, Y^-) = 0,
\end{align}  
in the same way we obtain $\chi(X^-, Y^+)=0$. Obviously $\chi(\zeta, 
X)=0$, $\chi(X,\zeta)=0$. We conclude $\chi = 0$ everywhere. Set 
$h'=-\psi h$, as $\tau$ is contact, by $R=\zeta$, we have $h'X = 
\tau(h'X)\zeta = 0$, simple considerations now follows $h=0$, hence 
$\mathcal M$ is normal.  Converse is obvious. If $\ml M$ is normal 
$h=0$, and $\chi=0$.  We now have $\Pi_\pm =0$, by (\ref{pg:inv}), 
therefore $\ml M$ is flat. 
\end{proof}
	
\begin{corollary}
Para-Sasakian manifold is flat as bi-Legendrian manifold.
\end{corollary}

\begin{corollary}
Along the proof we have established that for $\tau$-normal manifold 
which satisfies assumptions of above theorem form $\chi$ is symmetric. 
We see that such manifold has a quite a lot of symmetric invariants, 
ie. symmetric $(0,2)$-tensor fields, metric is evident others are 
$d\tau(\psi X,Y)$, $g(\bar h X,Y)$, and $\chi(X,Y)=d\tau(h X,\psi Y)$.
\end{corollary}

In following lemma we state basic properties of tensor field $\chi$.
\begin{proposition}
Let $\ml M$ be almost paracontact metric manifold with contact 
characteristic form and parallel characteristic vector field. Set 
$\chi(X,Y) = d\tau(hX,\psi Y)$. The tensor field $\chi$ satisfies 
following identities
\begin{align}
  & \chi(X,Y) = \chi(Y,X), \\
  \label{chi:dtau}
  & \chi(X,Y) = -d\tau(\psi hX,Y) = d\tau(h\psi X,Y), \\
  & \chi(X^+,Y^-) = 0, \quad \text{in words} \; \ml V^+, 
        \ml V^-  \; \text{are $\chi$-orthogonal}, \\
  & \chi|_{\ml V^{+}} = \tfarc{1}{2}\Pi_{+}, \quad 
    \chi|_{\ml V^{-}} = \tfarc{1}{2}\Pi_{-}, 
\end{align}
\end{proposition}

Let consider case where one of $\Pi_{+}$ or $\Pi_{-}$, vanishes. Such
bi-Legendrian manifold we call semi-flat. On base of established 
properties of form $\chi$ we can state following proposition.
\begin{theorem}
	Let $\ml M$, be almost paracontact metric manifold with contact 
	characteristic form, autoparallel characteristic vector field. 
	$\ml M$ is semi-flat if and only if $h^{2}=0$.
\end{theorem}
\begin{proof}
We already know that Legendrian distributions $\ml V^+$ and $\ml V^-$,
are $\chi$-orthogonal. From this $\chi$-orthogonal decomposition follows 
that if $\Pi_+ = 2\chi |_{\ml V^+}=0$, or $\Pi_-=2\chi |_{\ml V^-}=0$, 
then kernel of $\chi$ is non-trivial it contains corresponding flat 
distribution.  Consider $\chi(X, hY)$. We will show that $\chi(X, hY)=0$, 
for every vector fields $X$, $Y$ if $\ml M$ is semi-flat. Then by symmetry of 
$\chi$ we find 
\begin{align}
  & \chi(X,hY) = \chi(hY,X) = -d\tau(\psi h^2X,Y)=0, 
\end{align}
from the last equation
we easily obtain $h^2=0$. Note for vector fields $X$, $Y$, such that 
$\tau(X) =\tau(Y)=0$, there are unique decompositions  
$X=X^+ +X^-$, $Y=Y^++Y^-$, hence
\begin{align}
  & \chi(X,hY) = \chi(X^+, (hY)^+) + \chi(X^-, (hY)^-)  = \\ 
  & \qquad  \chi(X^+, hY^-)+ \chi(X^-,hY^+),
  \nonumber    
\end{align}
here we have used $(hY)^+ = hY^-$, and $(hY)^- = h Y^+$, which 
comes from the fact that $h$ and $\psi$ anti-commute $\psi h+h\psi =0$. 
To fix attention let $\Pi_+ =0$. Then above identity simplifies to 
$\chi(X,hY)=\chi(X^-, hY^+)$. From (\ref{chi:dtau}), 
\begin{align}
  & \chi(X^-, hY^+) = -d\tau(\psi h X^-, hY^+) = 
  -d\tau(\psi hY^+, hX^-) = \\ 
  & \qquad \chi(hX^-, Y^+) = 0, \nonumber
\end{align}
as $Y^{+}$ belongs to kernel of $\chi$. Thus $\chi(X,hY) = 0$, for 
every vector fields $X$, $Y$. 
\end{proof}

For  paracontact metric $(\kappa,\mu)$-space, curvature $R_{XY}\zeta$, 
satisfies   
\begin{align}
	R_{XY}\zeta = \kappa (\tau(Y)X - \tau(X)Y) + 
	  \mu(\tau(Y)hX - \tau(X)hY),
\end{align}
and $h^{2}X=(1+\kappa)(X-\tau(X)\zeta)$. In case $\kappa = -1$, 
$h^{2}=0$. 
\begin{corollary}
	Non-para-Sasakian paracontact metric $(-1,\mu)$-space is non-flat, 
	semi-flat bi-Legendrian manifold.
\end{corollary}


%
%

%
%
%
%
%
%
%


\end{document}